\documentclass[12pt,a4paper,reqno]{amsart}

\usepackage{soul}
\usepackage{amsmath,enumitem, comment, microtype}
\usepackage{graphicx}
\usepackage{xcolor}
\usepackage{ulem}

\usepackage{amssymb,amsthm}
\usepackage{mathrsfs}
\usepackage{bbm}

\def\R{\mathbb R}
\def\T{\mathbb T}

\usepackage{a4wide}
 \newtheorem{theo}{Theorem}[section]
 \newtheorem{cor}[theo]{Corollary}
 \newtheorem{lem}[theo]{Lemma}

 \newtheorem{defi}[theo]{Definition}
 \newtheorem{rem}[theo]{Remark}

\newtheorem{nota}[theo]{Notation}

\allowdisplaybreaks[2]

\title{Hausdorff dimension of recurrence sets for matrix transformations of tori}
\author{Zhang-nan Hu}
\address{Z.-N.~Hu, College of Science, China University of Petroleum, Beijing 102249, P. R. China}
\email{hnlgdxhzn@163.com}

\author{Bing Li*}\thanks{* Corresponding author} 
\address{B.~Li, School of Mathematics, South China University of Technology, Guangzhou, 510641, P. R. China}
\email{ scbingli@scut.edu.cn}

\date{\today}

\begin{document}

\begin{abstract}
 Let $T\colon\mathbb{T}^d\to \mathbb{T}^d$, defined by $T x=Ax(\bmod 1)$, where $A$ is a $d\times d$ integer matrix with eigenvalues $1<|\lambda_1|\le|\lambda_2|\le\dots\le|\lambda_d|$. We investigate the Hausdorff dimension of the recurrence set
 \[R(\psi):=\{x\in\mathbb{T}^d\colon T^nx\in B(x,\psi(n)) {\rm ~for~infinitely~ many~}n\}\]
 for $\alpha\ge\log|\lambda_d/\lambda_1|$, where $\psi$ is a positive decreasing function defined on $\mathbb{N}$ and its lower order at infinity is $\alpha=\liminf\limits_{n\to\infty}\frac{-\log \psi(n)}{n}$. In the case that $A$ is diagonalizable over $\mathbb{Q}$ with integral eigenvalues, we obtain the dimension formula. % for $\alpha>0$.
\end{abstract}

\maketitle

% I think the following results also hold for some integer matrix
% A with \det A=1 and an eigenvalues $\lambda>1$, but here we
% consider a special case first.

\section{Introduction}

Let $(X, \mathscr{B}, T, \mu, \rho)$ be a probability measure preserving system with a compatible metric $\rho$.  We call $(X, \mathscr{B}, T, \mu, \rho)$ a metric measure preserving system (m.m.p.s.). If $(X,\rho)$ is  separable, the Poincaré recurrence theorem shows that $\mu$-a.e. $x\in X$ is recurrent, that is
\[\liminf_{n\to\infty} \rho(T^nx, x) = 0.\]
It shows nothing about the speed at which the orbit returns close to the initial point.  One of the first general quantitative recurrence results was given by Boshernitzan \cite{bos}.
\begin{theo}[\cite{bos}]\label{bos}
  Let $(X, \mathscr{B},T, \mu,\rho)$ be a m.m.p.s. Assume that for
  some $\tau>0$, the $\tau$-dimensional Hausdorff measure
  $\mathcal{H}^{\tau}$ of $X$ is $\sigma$-finite. Then for
  $\mu$-a.e.\ $x\in X$,
  \[
    \liminf_{n\to\infty}n^{\frac{1}{\tau}}\rho(T^nx,x)< \infty.
  \]
  Futhermore, if $\mathcal{H}^{\tau} (X)=0$, then for $\mu$-almost
  every $x\in X$,
  \[
    \liminf_{n\to\infty}n^{\frac{1}{\tau}}\rho(T^nx,x)=0.
  \] 
\end{theo}

Later, Barreira and Saussol \cite{BS} showed that the exponent $\tau$ in Theorem \ref{bos} could be replaced by the lower local dimension of a measure at $x$. This leads us to study the size of recurrence set when the rate of recurrence is replaced by a general function. Define recurrence sets  as 
\[
  R(\psi) = \{\, x\in X : T^n(x) \in B(x,\psi(n)) \text{ for infinitely   many } n\ge1 \, \},
\]
where %$\{r_n\}$ is a sequence of positive real numbers.
$\psi\colon \mathbb{N} \to \mathbb{R}^+$ is a positive  decreasing function, and $B(x,\psi(n))$ denotes the ball centred at  $x$ with radius $\psi(n)$.

For  the measure of $R(\psi)$, 
Chang, Wu and Wu \cite{CWW} investigated the case when $X$ is a homogeneous self-similar set satisfying strong separation condition, and obtain results on the Hausdorff measure of $ R(\psi)$.  Similar results were generalised to finite conformal iterated function system satisfying the open set condition by Baker and Farmer\cite{BF}. Later, Hussain et.al \cite{hlsw} considered more general conformal dynamical systems. When $T$ is a piecewise expanding map, under some conditions, He and Liao \cite{hl} obtained that the measure of $R(\psi)$ obeys a full-zero law.  More results about measure of $ R(\psi)$ can be found in papers by \cite{ abb, kz,kkp}.
As for the size of  $ R(\psi)$ in Hausdorff dimension, Tan and Wang \cite{TW} calculated the Hausdorff dimension of $ R(\psi)$ when $T$ is the $\beta$-transformation. Seuret and Wang \cite{SW} proved similar results for self-conformal sets. There are very few results on the Hausdorff dimension when $T$ is the matrix transformation,  and we can refer to \cite{hl} for diagonal matrix transformations, and to \cite{HuPer} for cat maps. It's also investigated in several special cases such as \cite{wy,yl}. 

Shrinking target problem concerns the speed at which  $\{T^nx\}_{n\ge1}$ returns to the neighborhood of a given point $x_0$ instead of the initial point $x$, which has many common features with the problem of quantitative recurrence. For the shrinking target problem, much more results are known.  Hill and Velani \cite{HV} investigated the Hausdorff dimension of shrinking target sets in the system $(X, T )$ with $X$ a $d$-dimensional torus, and $T$ a linear map given by an integer matrix.  For a  real, non-singular matrix transformation, under some conditions,  Li %, Liao, Velani and Zorin 
et al. \cite{llvz} proved that the Lebesgue measure of the shrinking target set obeys a zero-one law. They also determined the Hausdorff dimension of  shrinking targets set When $T$ is a diagonal matrix transformation. One can refer to \cite{CK, FMP} for more results on the measure, and to \cite{HV95, lwwx} for the dimension.

%\sout{For recurrence sets, there are very few results on the Hausdorff dimension when $T$ is the matrix transformation, and we can refer to} \cite{hl,HuPer} \sout{for some concrete cases.} 
Motivated by the aforementioned research, in this paper, we focus on the case where $X$ is the $d$-dimensional torus $\mathbb{T}^d$ endowed with the usual quotient distance $\rho_0$,  and $T$ is the integer matrix transformation with the modulus of eigenvalues are strictly larger than 1. More precisely, $T\colon\mathbb{T}^d\to \mathbb{T}^d$ is defined by
$$Tx=Ax\pmod 1,$$ where $A%= \begin{bmatrix} a&b\\c&d\end{bmatrix}
$ is  a $d\times d$ integer matrix. Let $\psi\colon \mathbb{N} \to \mathbb{R}^+$ be a positive function satisfying $\psi(n)\to0$ as $n\to\infty$. Throughout, denote  the lower order of $\psi$ at infinity by
\[
  \alpha:= \liminf_{n\to\infty} \frac{-\log \psi(n)}{n}.
\]

%\sout{Denote the Lebesgue measure restricted on $\T^d$ by $\mathcal{L}$, and }
In this paper, $\dim_{\rm H}$ stands for the Hausdorff dimension.
 
 \begin{theo}\label{theorem1}
 Let $A$ be a $d\times d$  integer matrix with all eigenvalues $\lambda_1,~\lambda_2,\dots,~\lambda_d$. Assume that $|\lambda_d|\ge\dots\ge|\lambda_1|>1$. Then 
 %{\color{red}\sout{ if $\alpha>\log(|\lambda_d|/|\lambda_1|)$,}
% \[\dim_{\rm H}R(\psi)\le\min_{j\in\{1,\dots,d\}}\Big\{\frac{j\log|\lambda_j|+\sum_{i=j+1}^d\log|\lambda_i|}{\alpha +\log|\lambda_j|}\Big\}.\]}
 for $\alpha\ge\log|\lambda_d/\lambda_1|$, 
 \[\dim_{\rm H}R(\psi)=\min_{j\in\{1,\dots,d\}}\Big\{\frac{j\log|\lambda_j|+\sum_{i=j+1}^d\log|\lambda_i|}{\alpha +\log|\lambda_j|}\Big\}.\]
 \end{theo}
 
 %{\color{red}\begin{theo}\label{theorem2*2}
% Let $A$ be a $d\times d$  integer matrix with all eigenvalues $\lambda_1,~\lambda_2$. Assume that $|\lambda_2|\ge |\lambda_1|>1$. Then 
% \[\dim_{\rm H}R(r_n)\le\min \Big\{\frac{\alpha+2\log|\lambda_1|}{\alpha+\log|\lambda_1|}, \frac{2\log|\lambda_2|}{\alpha+\log|\lambda_2|}, \frac{\log|\lambda_1\lambda_2|}{\alpha+\log|\lambda_1|}\Big\}.\]
% \end{theo}}
\begin{rem}
%{\color{red}\sout{Under the assumption} 
Let $A$ be as in Theorem \ref{theorem1}. For $\alpha\ge0$, we always have
\[\dim_{\rm H}R(\psi)\le\min_{j\in\{1,\dots,d\}}\Big\{\frac{j\log|\lambda_j|+\sum_{i=j+1}^d\log|\lambda_i|}{\alpha +\log|\lambda_j|}\Big\},\]
here we do not need assume that $\alpha\ge\log|\lambda_d/\lambda_1|$.  Notice that for $\alpha>0$, $\dim_{\rm H}R(\psi)\le \frac{d\log|\lambda_d|}{\alpha +\log|\lambda_d|}<d$.
\end{rem}

As an immediate consequence of Theorem \ref{theorem1}, when the modulus of all eigenvalues of $A$ are same, we obtain the formula of the Hausdorff dimension of $R(\psi)$ for any $\alpha\ge 0$.
\begin{cor}
 Let $A$ be a $d\times d$  integer matrix with all eigenvalues  $\lambda_1,~\lambda_2,\dots,~\lambda_d$. Assume that the modulus of all eigenvalues are same, denoted by $\lambda$, and $\lambda>1$.  Then  
 %{\color{red}\sout{ if $\alpha>\log(|\lambda_d|/|\lambda_1|)$,}
% \[\dim_{\rm H}R(\psi)\le\min_{j\in\{1,\dots,d\}}\Big\{\frac{j\log|\lambda_j|+\sum_{i=j+1}^d\log|\lambda_i|}{\alpha +\log|\lambda_j|}\Big\}.\]}
 for $\alpha\ge 0$, 
 \[\dim_{\rm H}R(\psi)=\frac{d\log\lambda}{\alpha +\log\lambda}.\]

\end{cor}

When an integer matrix $A$ is diagonalizable over $\mathbb{Z}$ or $A$ is a  diagonal real matrix with eigenvalues $\lambda_d\ge\dots\ge\lambda_1>1$, Yuan and Wang \cite{yw} calculated the Hausdorff dimension of $R(\psi)$.
%\begin{equation}\label{wy}
%\dim_{\rm H}R(\psi)=\min_{j\in\{1,\dots,d\}}\Big\{\frac{j\log|\lambda_j|+\sum_{k\in\mathcal{K}(j)}(\alpha+\log|\lambda_j|-\log|\lambda_i|)+\sum_{i=j+1}^d\log|\lambda_i|}{{\log|\lambda_j|+\alpha}}\Big\}.
%\end{equation}
%Notice that if $\alpha>\log|\lambda_d/\lambda_1|$, we have $\bigcup_j\mathcal{K}(j)=\varnothing$, hence in this case, the formula of Hausdorff dimension in Theorem \ref{theorem1} coincides with \eqref{wy}. 
However they did not consider the case when $\lambda_i<-1$ for some $1\le i\le d$. When $A$ is a  diagonal real matrix with $|\lambda_d|\ge\dots\ge|\lambda_1|>1$, He and Liao \cite{hl} gave the formula of $\dim_{\rm H}R(\psi)$. If $A$ is an integer matrix, the diagonal assumption can be relaxed to a weaker condition as the following theorem shows. 

\begin{theo}\label{theorem2}
  Let $A$ be a $d\times d$  integer matrix. % with all eigenvalues $\lambda_1,~\lambda_2,\dots,~\lambda_d$.
    Assume that $A$ is diagonalizable over $\mathbb{Q}$ with all eigenvalues $\lambda_1,~\lambda_2,\dots,~\lambda_d\in\mathbb{Z}$. Assume that $|\lambda_d|\ge\dots\ge|\lambda_1|>1$. Then  for $\alpha\ge 0$,
\[\dim_{\rm H}R(\psi)=\min_{j\in\{1,\dots,d\}}\Big\{\frac{j\log|\lambda_j|+\sum_{k\in\mathcal{K}(j)}(\alpha+\log|\lambda_j|-\log|\lambda_i|)+\sum_{i=j+1}^d\log|\lambda_i|}{{\log|\lambda_j|+\alpha}}\Big\},\]
where 
\[\mathcal{K}(j):=\{i\in\{1,\dots,d\}\colon \log|\lambda_i|>\log|\lambda_j|+\alpha\}.\]
%\[\mathcal{K}_2(j):=\{k\colon \log|\beta_k|\le\log|\beta_j|\}\]
%\[\mathcal{K}_3(j):=\{1,\dots,d\}\setminus (\mathcal{K}_1(j)\bigcup \mathcal{K}_2(j))\]
 \end{theo}
 
 \begin{rem}
 The dimension formula in Theorem \ref{theorem2} is same as those given by \cite{hl,yw}.
When $\alpha\ge\log|\lambda_d/\lambda_1|$, $\bigcup_j\mathcal{K}(j)=\varnothing$, hence in this case, the formula of Hausdorff dimension in Theorem \ref{theorem1} coincides with that given in Theorem \ref{theorem2}. 
 \end{rem}

%Rewrite 
%\begin{align*}
 % R(r_n) &= \{\,x\in\T^d : (A^n-I)x \pmod 1 \in B(0,r_n) \, i.o.\},\\%=
                                %\bigcup_{i=1}^{M_n}
                                %\mathcal{E}_n^i, \]
  %    &= \bigcap_{k=1}^{\infty}\bigcup_{n=k}^{\infty}\{x\in\T^d : (A^n-I)x \pmod 1 \in B(0,r_n) \, \},
%\end{align*}
 For $n\ge1$, write
\[R_n(\psi)=\{x\in\T^d : (A^n-I)x \pmod 1 \in B(0,\psi(n)) \, \},\]
%hence $R= \limsup\limits_{n\to\infty} R_n$. 
where $I$ is the identity matrix. Then $R(\psi)=\limsup\limits_{n\to\infty}R_n(\psi)$.

The paper is organized as follows. %Theorem \ref{theorem2} is proved in Section 2. 
The Theorem \ref{theorem1} is proved in the following two sections.  Section 2 is devoted to give some preparations on the geometry property of $R_n(\psi)$, which is crucial to the proof of Theorem \ref{theorem1}. In Section 3, we use the preparations to construct a Cantor subset to establish the lower bound of $\dim_{\rm H}R(\psi)$. We only prove Theorem \ref{theorem1} in the case that $|\lambda_i|>1$, $1\le i\le d$, otherwise we can consider $A^2$ instead of $A$. In the last section, we give the proof of Theorem \ref{theorem2}.

\begin{nota}
Write $f\lesssim g$ if $f\le Cg$ for some constant $C>0$, and $f\asymp g$ if $f\lesssim g$ and $g\lesssim f$.
The ceiling function of a real number $x$ is denoted by $\lceil x \rceil$, and the floor function is denoted by $\lfloor x\rfloor$.
For a matrix $A$, let $\det A$ stand for the determinant of  $A$. We write $\# E$ for the cardinality of a finite set $E$.
\end{nota}

%{\color{blue}
%\section{The proof of Theorem \ref{theorem2*2}}
%Assume that $\lambda_2> \lambda_1>1$. 
%If $\alpha\ge\log\lambda_2-\log \lambda_1$, the upper bound follows from Theorem \ref{theorem1}.
%Now we consider the case when $\alpha <\log\lambda_2-\log \lambda_1$. }

%In the following sections, we will give the proof of Theorem \ref{theorem1}.

\section{%Basic information on $R_n(\psi)$
Distribution of periodic points}
Throughout this paper, let the eigenvalues of the matrix $A$ be $\{\lambda_1,\lambda_2,\dots,\lambda_d\}$ with $|\lambda_1|\le |\lambda_2|\le\cdots\le|\lambda_d|$.
\begin{defi}\label{perpoint}
A point $x\in\T^d$ is called a {\bf periodic point}
with period $n$ if
\begin{equation} \label{equation:periodic}
  (A^n-I)x~(\rm mod~1)=0.
\end{equation}
\end{defi}
It is easy to see that \eqref{perpoint} is equivalent to $A^nx=x~(\rm mod ~1)$.
Put
$$\mathcal{P}_n=\{x\in\T^d\colon ~x %(A^n-I)x~(\rm mod~1)=0
\text{ ~is a periodic point with period~} n\}.$$
%\sout{The following lemma will give $\#\mathcal{P}_n$, that is, the number of periodic points with period $n$. }

\begin{lem}[Lemma 2.3 in \cite{EW99}]\label{number-periodic}
%Let $A$ be an integer matrix with no eigenvalue being a root of unity, thenfor $n\ge1$, 
Let $M$ be a non-singular integer matrix. If no eigenvalue of $M$ is a root of unity, the number of periodic points with period $n$ is given by
\[\#\{x\in\mathbb{T}^d\colon M^n x({\rm mod~1})=x\}=|\det (M^n-I)|.\] 
%$$%\#\{ x\in\T^d\colon (A^n-I)x~(\rm mod~1)=0\}
%\#\mathcal{P}_n=%{\it H_n}=
%|\det (A^n-I)|.$$
\end{lem}
Denote $|{\rm det} (A^n-I)|$ by $H_n$. It follows from Lemma \ref{number-periodic} that
  $$\#\mathcal{P}_n=H_n=\prod_{{\it i}=1}^{\it d}|{\it \lambda_i^n}-1|.$$

Rewrite
\begin{equation}\label{trans}
\begin{split}
R_n(\psi) &=\{x\in\mathbb{T}^d\colon x\in(A^n-I)^{-1}B(0,\psi(n))+(A^n-I)^{-1}\mathbb{Z}^d\}\\
&=\bigcup_{y\in\mathcal{P}_n}\{x\in\mathbb{T}^d\colon x\in(A^n-I)^{-1}B(0,\psi(n))+y\}.
\end{split}
\end{equation}

By \eqref{trans} and Lemma \ref{number-periodic}, $R_n(\psi) $ consists of $H_n$ ellipsoids which are translations of $(A^n-I)^{-1}B(0,\psi(n))$, denoted by $\{R_{n,i}\}_{i=1}^{H_n}$.
%\sout{Kirsebom, Kunde and Persson\cite{kkp} proved it in the proof of Theorem D.}
For $n\ge1$ and $j=1,2,\dots,d$, put $$\ell_{n,j}=2\psi(n)|\lambda_j^n-1|^{-1}.$$ Let $e_{n,1} \ge e_{n,2}\ge\dots\ge e_{n,d}$ be the lengths of semi-axes of the ellipse $R_{n,i}$. If $A$ is diagonalisable, for $j=1,\dots,d$
$$e_{n,j}\asymp\ell_{n,j}.$$
 If $A$ is not diagonalisable, 
using Jordan decomposition, we have the following lemma.
\begin{lem}\label{singular and eigen}
Let $A$ be a non-singular integer matrix. Assume that  the modulus of all eigenvalues are not 1. Then there are constants $C > 1$ and $\tau>0$ such that
\[C^{-1}n^{-\tau} \le \frac{e_{n,j}}{\ell_{n,j}}  \le C n^{\tau}\]
holds for all $n$ and $1\le j\le d$.
\end{lem}

\begin{proof}
%The proof is a slight modification of Lemma 2 and Lemma 3 in \cite{HV}.
% Without loss of generality, assume that $\lambda_i\ge0$, $i=1,2,\dots,d$.
%  Considering Jordan decomposition of $A$, there is a non-singular matrix $D$ such that
% \[A=D^{-1}\rm{diag}\{A_1,\dots,A_s\}D,\]
% where $A_i$ is the Jordan matrix with respect to eigenvalue $\lambda_{k_i}$ (duplicates possible). Notice that $A_i=\lambda_{k_i}\widetilde{A}_i$, here $\widetilde{A}_i$ is a upper triangular matrix whose diagonal elements  are  1 and all super-diagonal elements are $\lambda_{k_i}^{-1}$. Put $C=\rm{diag}\{\widetilde{A}_1,\dots,\widetilde{A}_s\}$.  Hence $A$ is similar to $BC$, here $B$ is  diagonal with elements $\{\lambda_{k_i}\}_{i=1}^s$. %It derives that $A^n$ is similar to $B^nC^n$, because $B$ is diagonal. 
 %Then $(A^n-I)^{-1}$ is similar to $C^{-n}(B^n-C^{-n})^{-1}$. Note that $(B^n-C^{-n})^{-1}=\widetilde{B}_n\widetilde{C}_n$, where $\widetilde{B}_n$ is  diagonal with elements $\{(\lambda_{k_i}^n-1)^{-1}\}_{i=1}^s$, and by \cite[Lemma 2]{HV}, $\widetilde{C}_n$ is a upper triangular matrix whose diagonal elements  are 1, and the entries are bounded by a polynomial in $n$.
By Lemma 3 in \cite{HV}, for $n\ge1$, we have $(A^n-I)^{-1}=A_1A_2$, where all eigenvalues of  $A_1$ have absolute value 1, the matrix $A_2$ is diagonalised over $\mathbb{R}$ with eigenvalues of modulus  $|\lambda_i^n-1|$, $1\le i\le d$, and $A_1$ and $A_2$ commute. 
Applying \cite[Lemma 2]{HV}, there is $\tau\ge0$ depending only on $A$ such that for any ball $B(x,r)$, 
 $$B(x'',O(n^{-\tau})r)\subset A_1B(x,r)\subset B(x',O(n^{\tau})r)$$
for some $x', ~x''\in\mathbb{T}^d$. Combining these, we conclude that $(A^n-I)^{-1}B(x,r)$ contains an ellipse with lengths of semi-axes $O(n^{-\tau})|\lambda_i^n-1|^{-1}r$, $1\le i\le d$, and also is contained in an an ellipse with lengths of semi-axes $O(n^{\tau})|\lambda_i^n-1|^{-1}r$, $1\le i\le d$.

\end{proof}

\begin{rem}\label{assumption}
By Lemma \ref{singular and eigen}, $e_{n,j}/\ell_{n,j}$ grows with polynomial speed, which does not influence the formula of Hausdorff dimension. 
%Hence from now on, we assume that $A$ is diagonalised for simplicity.  Then the lengths of semi-axes of the ellipse $R_{n,i}$ are comparable to $\ell_{n,j}$, $j=1,\dots,d$. Without loss of generality, we also assume that $\lambda_j>1$, $j=1,\dots,d$, and  $r_n=e^{-\alpha n}$, $n\ge1$ in the rest of the paper.
Hence for simplicity, from now on, we assume that $A$ is diagonalised over $\mathbb{R}$ with eigenvalues $\lambda_i>1$, $1\le i\le d$, and $r_n=e^{-\alpha n}$, $n\ge1$. 
\end{rem}

Recall that $\rho_0$ is the usual quotient distance. For $i\ne k\in\{1,2,\dots,H_n\}$, denote
\[d_n(i,k)=\inf\big\{\rho_0(x,y)\colon x\in R_{n,i},~y\in R_{n,k}\big\},\] 
that is, $d_n(i,k)$ is the distance between $R_{n,i}$ and $R_{n,k}$. When $i=k$, we have $d_n(i,k)=0$. Put
\[d_n=\min_{ 1\le i\ne k\le H_n}d_n(i,k).\]
The following lemma gives some informations on $d_n$, the shortest distance between the ellipses. 

\begin{lem}\label{distance}
%If $A$ is an integer matrix such that the modulus of all eigenvalues are not 1,  
%\begin{enumerate}%[(1)]
%{\color{red}Let $A$ be as  in Theorem \ref{theorem1}. Then}
Let $A$ be a non-singular integer matrix. Assume that  the modulus of all eigenvalues are not 1. Then for $n$ large enough,
\[d_n\gtrsim (\lambda_d^n-1)^{-1}.\]
In particular, we have 
$R_{n,i}\cap R_{n,j}=\varnothing$
 for $i\ne j$, $1\le i,j\le H_n$.

%\begin{enumerate}
%\item {\color{red}for $n\ge1$, $\{R_{n,i}\}_{i=1}^{H_n}$ do not intersect each other.}
%\item for $n\ge1$, let $d_n$ be  the shortest distance between $\{R_{n,i}\}_{i=1}^{H_n}$. Then 
%$$d_n\gtrsim\min\limits_{j=1,2,\dots,d}\{|\lambda_j^n-1|^{-1}\}.$$
%\item  \sout{for $n\ge1$,  $\mathcal{L}(R_n(\psi))=\pi \psi(n)^d$, and for any $i\in\{1,2,\dots,H_n\}$, $\mathcal{L}(R_{n,i})=\pi\frac{\psi(n)^d}{H_n}$.}\footnote{This conclusion didn't be used in the following paper, so i think we can delete it.}
%\end{enumerate}
%\end{enumerate}
\end{lem}

\begin{proof}
Without loss of generality, we assume that $\psi(1)<\frac{1}{3}$. For $n\ge1$ and $x\in\mathcal{P}_n$, note that
\[B\big((A^n-I)x,\psi(n)\big)\subset B\big((A^n-I)x,\frac{1}{3}\big),\]
and $\{B\big((A^n-I)x,\frac{1}{3}\big)\colon x\in\mathcal{P}_n\}$ are disjoint, since $\{(A^n-I)x\colon x\in\mathcal{P}_n\}\subset \mathbb{Z}^d$. It follows that  $\{T^{-n}B\big((A^n-I)x,\frac{1}{3}\big)\colon x\in\mathcal{P}_n\}$ are also disjoint.
Then we conclude that  for $1\le i \le H_n$, the ellipse $R_{n,i}$ is contained in an ellipse $\widetilde{R}_{n,i}$ with lengths of semi-axes $\frac{1}{3}(\lambda_j^n-1)^{-1}$, $1\le j \le d$, and $\{\widetilde{R}_{n,i}\}_i$ are disjoint. It implies that for any $i\ne k$,
\[d_n(i,k)\ge \min\limits_{j=1,2,\dots,d}\Big\{2\Big(\frac{1}{3}-\psi(n)\Big)(\lambda_j^n-1)^{-1}\Big\},\]
which implies that $d_n\ge 2(\frac{1}{3}-\psi(n))(\lambda_d^n-1)^{-1}>0 $. 

%\sout{(2)\, By \cite{entropy}, $T$ preserves Lebesgue measure, that is,}
%\[\mathcal{L}(R_n(\psi))=\mathcal{L}(B(0,\psi(n))=\pi \psi(n)^d.\]
%\sout{It follows from (1) that }
%$$\mathcal{L}(R_{n,i})=\frac{\mathcal{L}(R_n(\psi))}{H_n}=\pi\frac{\psi(n)^d}{H_n}.$$
\end{proof}

Recall $\mathcal{P}_n$ consists of  all periodic points with period $n$. % and by Lemma 2.2, $\#\mathcal{P}_n=H_n$. 
 Now we estimate the number of  periodic points with period $n$ in a given ball.

\begin{lem}\label{local counting}
For any $B=B(x,r)$ in $\T^d$ and $n\ge1$, we have
\[\#\mathcal{P}_n\cap B:=\#\{y\in B\colon (A^n-I)y\pmod 1=0\}\lesssim \prod_{j\colon (\lambda_j^n-1)r>1}\lceil(\lambda_j^n-1)r\rceil.\]  
If $r(\lambda_1^{n}-1)>1$, then 
%{\color{red}If $A$ is symmetric and eigenvales are irrational}, for $n$ large enough (depends on $r$), we have
\[\#\mathcal{P}_n\cap B\asymp r^dH_n.\]
\end{lem}
\begin{proof}
%For the first part, 
%Use the proof of Lemma \ref{cover and contain}.
 For $n\ge1$, 
$$\#\mathcal{P}_n\cap B=\#\{(A^n-I)B\cap \mathbb{Z}^d\}.$$
Notice that there is a constant $c_1>1$ such that $(A^n-I)B$ may be covered by a rectangle $R_1$ with length $2c_1(\lambda_j^n-1)r$, $j=1,\dots,d$, and contains a rectangle $R_2$ with length $\frac{1}{c_1\sqrt{d}}(\lambda_j^n-1)r$, $j=1,\dots,d$.

(i)\, A square of length $1/2$ contains at most one point in $\mathbb{Z}^d$, and  $R_1$ can be covered by 
\[\prod_{j\colon 2c_1(\lambda_j^n-1)r>1/2}\lceil 4c_1(\lambda_j^n-1)r \rceil\]
squares with length $1/2$, which implies that 
\begin{equation}\label{num}
\#\mathcal{P}_n\cap B\lesssim \prod_{j\colon (\lambda_j^n-1)r>1}\lceil (\lambda_j^n-1)r \rceil.
\end{equation}

(ii)\, If $r(\lambda_1^{n}-1)>1$, then $r(\lambda_j^{n}-1)>1$, $j=1,\dots,d$.

%Now we estimate the lower bound on $\#\mathcal{P}_n\cap B$. 
Since a square of length 1 contains at least one point in $\mathbb{Z}^d$, it suffices to estimate the number of squares with length 1 which are contained in $R_2$.  Note that $R_2$ contains 
\[\prod_{j=1}^d\Big\lfloor\frac{1}{c_1\sqrt{d}}(\lambda_j^n-1)r\Big\rfloor\]
squares with length 1. Combing with \eqref{num}, we have 
\[\#\mathcal{P}_n\cap B\asymp r^d\prod_{j=1}^d(\lambda_j^n-1)=r^dH_n.\]

%For the second part, note that $\#\mathcal{P}_n\cap B=\#(A^n-I)B\cap \mathbb{Z}$, and $(A^n-I)B$ can contain and be contained in some rectangle $R(n)$ whose lengths are comparable to $(1-\lambda_1^n)r$ and $(\lambda_2^n-1)r$. Here we assume that $y_2=ky_1$ is the line on which the longer semi-axis of $(A^n-I)B$ lies.  Here $k\in\mathbb{R}\setminus \mathbb{Q}$. Then for some $a\ge0$. 
%\begin{equation*}
%\begin{split}
%\#R(n)\cap \mathbb{Z}&=\Big\{x\in \Big[a, a+\frac{2}{\sqrt{1+k^2}}(\lambda_2^n-1)r\Big]\cap \mathbb{Z}\colon ||kx||<2\sqrt{1+k^2}(1-\lambda_1^n)r\Big\}\\
%\end{split}
%\end{equation*}
%For $n$ large enough, $\sqrt{1+k^2}r\le 2\sqrt{1+k^2}(1-\lambda_1^n)r\le 2\sqrt{1+k^2}r$. Hence
%\begin{equation*}
%\begin{split}
%\begin{multline*}
%\Big\{x\in \Big[a, a+\frac{2}{\sqrt{1+k^2}}(\lambda_2^n-1)r\Big]\cap \mathbb{Z}\colon ||kx||<\sqrt{1+k^2}r\Big\}\subset \\
%R(n)\cap \mathbb{Z}\\
%\subset \Big\{x\in \Big[a, a+\frac{2}{\sqrt{1+k^2}}(\lambda_2^n-1)r\Big]\cap \mathbb{Z}\colon ||kx||<2\sqrt{1+k^2}r\Big\}.
%\end{multline*}
%\end{split}
%\end{equation*}
%Here $||kx||=kx-\lfloor kx\rfloor$. Since $(kx)_{x\in\mathbb{Z}_+}$ is uniformly distributed modulo 1, for $N$ large enough, we have 
%\begin{multline*}
%\#R(n)\cap \mathbb{Z}=
%\[ A([0,cr]; N; \omega)\asymp r N.  \]
%\end{multline*}
%Hence 
%$$\#R(n)\cap \mathbb{Z}\asymp (1-\lambda_1^n)(\lambda_2^n-1)r^2=r^2|H_n|. $$

\end{proof}
\begin{rem}\label{remark}
For any $B(x,l)$ in $\T^d$, given $r>0$, if $2(\lambda_d^n-1)l\le r$, then $(A^n-I)B(x,l)$ can be covered by only one ball of radius $r$. Hence in Lemma \ref{local counting}, if $(\lambda_d^n-1)l\le 1/4$, we have
$$\#\{y\in B(x,l)\colon (A^n-I)y\pmod 1=0\}\le1.$$

\end{rem}
The following corollary follows from Lemma \ref{local counting},  which is crucial to give the lower bound on $\dim_{\rm H}R(\psi)$. %Let $B_n:= (A^n-I)^{-1}$. 
Recall that $R_n(\psi)=\bigcup_{i=1}^{H_n}R_{n,i}$.
\begin{cor}\label{cor}
For $n\ge1$, and $i\in\{1,\dots,H_n\}$, %the ellipsoid $R_{n,i}$ contains (comparable to) $r_n^dH_mH_n^{-1}$ points in $\mathcal{P}_m$ for $m\gg n$ large enough, that is,
\[\#\mathcal{P}_m\cap R_{n,i}\asymp \psi(n)^dH_mH_n^{-1}\]
holds for $m\ge n$ large enough.
\end{cor}
\begin{proof}
Recall that $\ell_{n,j}=2\psi(n)(\lambda_j^n-1)^{-1}$, $j=1,\dots,d$. For $n\ge1$, and $i\in\{1,\dots,H_n\}$, $R_{n,i}$ contains a rectangle $\widetilde{R}_{n,i}$ with length $\frac{1}{\sqrt{d}}\ell_{n,j}$, $j=1,\dots,d$, and $\widetilde{R}_{n,i}$ contains
\[\prod_{j=1}^d\Big\lceil \frac{\lambda_d^n-1}{\lambda_j^n-1} \Big\rceil\]
disjoint squares with length $\frac{1}{\sqrt{d}}\ell_{n,d}$.

Taking $m>n$ large enough such that $\frac{1}{\sqrt{d}}\ell_{n,d}(\lambda_1^m-1)>1$, by Lemma \ref{local counting}, for any ball $B$ with radius $\frac{1}{2\sqrt{d}}\ell_{n,d}$, we have
\[\#\mathcal{P}_m\cap B\asymp \ell_{n,d}^dH_m.\]
Therefore 
\begin{equation*}
\begin{split}
\#\mathcal{P}_m\cap R_{n,i}&\gtrsim \ell_{n,d}^dH_m\prod_{j=1}^d\lceil \frac{\lambda_d^n-1}{\lambda_j^n-1} \rceil\asymp H_m\Big(\frac{\psi(n)}{\lambda_d^n-1}\Big)^d\prod_{j=1}^d\lceil \frac{\lambda_d^n-1}{\lambda_j^n-1} \rceil\\
&\asymp \psi(n)^dH_m\prod_{j=1}^d\frac{1}{\lambda_j^n-1} =\psi(n)^dH_mH_n^{-1}.
\end{split}
\end{equation*}
Note that $R_{n,i}$ is contained a rectangle with length $\ell_{n,j}$, $j=1,\dots,d$. It follows that
\[\#\mathcal{P}_m\cap R_{n,i}\lesssim \psi(n)^dH_mH_n^{-1}.\]
\end{proof}

 In the following, we prove Theorem \ref{theorem1}. The proof  is divided into two parts: Section \ref{up} and Section \ref{low}. In Section \ref{up}, we use natural covering to estimate the upper bound on $\dim_{\rm H}R(\psi)$, and in Section \ref{low}, we  will construct a Cantor subset $K$ of $R(\psi)$ to  give a lower bound to $\dim_{\rm H}R(\psi)$.

%\section{The proof of Theorem \ref{theorem1}}
\section{Upper bound on $\dim_{\rm H}R(\psi)$}\label{up}

In Remark \ref{assumption}, we assumed that $A$ is diagonalised, hence there is a constant $c_2>1$ such that the quotient of singular values and eigenvalues of $(A^n-I)^{-1}$ is bounded by $c_2$ from above, and $c_2^{-1}$ from below for $n$ large enough.
%\subsection{Upper bound on $\dim_{\rm H}R(\psi)$}

For $m\ge1$, we have $R(\psi)\subset \bigcup_{n=m}^{\infty}R_n(\psi)=\bigcup_{n=m}^{\infty}\bigcup_{i=1}^{H_n}R_{n,i}$. Recall that the lengths of semi-axes of $R_{n,i}$ are about $\ell_{n,j}=2\psi(n)(\lambda_j^n-1)^{-1}$, $j=1,2,\dots,d$.

For $k\in\{1,2,\dots,d\}$,  we may use balls with radius $\ell_{n,k}$ to cover $R_{n,i}$, and the number of such balls is about 
\begin{equation}\label{num}
\prod_{j<k}\frac{\ell_{n,j}}{\ell_{n,k}}=\prod_{j=1}^k\frac{\lambda_k^n-1}{\lambda_j^n-1}.
\end{equation}
%\sout{balls of radius $\lambda_{n,k}$, since $1<\lambda_1\le \dots\le \lambda_d$.}

For any $\delta>0$, 
\begin{equation*}
\begin{split}
\mathcal{H}_{\delta}^s(R(\psi))%&%\le\sum_{n=m}^{\infty}\mathcal{H}_{\delta}^s(R_n)
%\le\sum_{n=m}^{\infty}\sum_{i=1}^{H_n}\mathcal{H}_{\delta}^s(R_{n,i})\le \sum_{n=m}^{\infty}\sum_{i=1}^{H_n}l_{n,k}\Big(\frac{r_n}{\lambda_k^n-1}\Big)^s\\
&\lesssim \sum_{n=m}^{\infty}H_n\ell_{n,k}^s\prod_{j=1}^k\frac{\lambda_k^n-1}{\lambda_j^n-1}\lesssim\sum_{n=m}^{\infty} \lambda_k^{n(k-s)}\psi(n)^s\prod_{j=k+1}^d\lambda_j^n\\
%&=\sum_{n=m}^{\infty}H_n\prod_{j=1}^k\frac{\lambda_k^n-1}{\lambda_j^n-1}\Big(\frac{r_n}{\lambda_k^n-1}\Big)^s\\
&=\exp\Big\{n\Big(k\log\lambda_k+\sum_{j=k+1}^d\log\lambda_j-s(\alpha+\log\lambda_k)\Big)\Big\}.
\end{split}
\end{equation*}
Notice that for any $s>\frac{k\log\lambda_k+\sum_{j=k+1}^d\log\lambda_j}{\alpha+\log\lambda_k}$, $\mathcal{H}^s(R(\psi))<\infty$, hence 
$$\dim_{\rm H}R(\psi)\le \min_k\Big\{\frac{k\log\lambda_k+\sum_{j=k+1}^d\log\lambda_j}{\alpha+\log\lambda_k}\Big\}.$$

\begin{rem}
% \sout{Given $0<r<{\color{red}\frac{1}{2}}$ and $y\in\mathcal{P}_n$,}
% \[\sout{y+(A^n-I)^{-1}B(0,\psi(n))\subset y+(A^n-I)^{-1}B(0,r),}\]
%\sout{and there is a set of measure 0 outside which $\{ y+(A^n-I)^{-1}B(0,r)\}_{y\in\mathcal{P}_n} $ do not intersect each other.} 
It follows from Lemma \ref{distance} that for $i\ne j$,
\[{\rm dist}(R_{n,i},R_{n,j})\gtrsim (\lambda_d^n-1)^{-1}.\]
When  $\alpha>\log(\lambda_d/\lambda_1)$, we have $\ell_{n,1}\lesssim (\lambda_d^n-1)^{-1}$, which implies any ball with radius $\ell_{n,k}$ intersects only one ellipsoid $R_{n,i}$. 
But if $\lambda_d/\lambda_1\ne1$, for  $\alpha\in\big(0, \log(\lambda_d/\lambda_1)\big]$, the number of balls of radius $\ell_{n,k}$ covering $R_n(\psi)$ may be less that the number in \eqref{num}, cause a ball of radius $\ell_{n,k}\le \ell_{n,1}$ can cover many $R_{n,i}$, that is,  there may exist  better coverings for $R_n(\psi)$, which gives another dimension formula such as Theorem \ref{theorem2}.
\end{rem}

\section{Lower bound on $\dim_{\rm H}R(\psi)$.}\label{low}

Recall that $R_n$ consists of elliptical discs $\{R_{n,i}\}_{i=1}^{H_n}$ with lengths of semi-axes $\ell_{n,j}$, $j=1,2,\dots,d$, whose centres are periodic points satisfying \eqref{equation:periodic}. We first suppose that $\log(\lambda_d/\lambda_1)<\alpha<\infty$.%{\color{red}Here $\ell_{n,j}=2\psi(n)(\lambda_j^n-1)^{-1}$.}

\subsection{Construct a Cantor set}

%{\color{red}False! For $n\ge1$, by Lemma \ref{rectangle}, there exists $\mathcal{A}_n\subset \{1,2,\dots,H_n\}$ such that 
%$$\{R_{n,i}\colon i\in\mathcal{A}_n\}\subset \{R_{n,i}\}_{i=1}^{H_n},$$
%$$5R_{n,i}\cap 5R_{n,j}=\varnothing,~i\ne j\in \mathcal{A}_n,$$}
Write $\mathcal{A}_n= \{1,2,\dots,H_n\}$ .
 Let $\{n_j\}_{j\ge0}$ be an increasing sequence of positive integers (to be determined later). 
 
Put $n_0=0$, $K_0=\T^d$.  Let 
\[K_1=\bigcup_{i\in \mathcal{C}_1}R_{n_1,i}, \]
where %$R_{n,i}=(A^n-I)^{-1}B(0,r_n)+y$ for some $y\in\mathcal{P}_n$, and 
\[\mathcal{C}_1:=\{i\in\mathcal{A}_{n_1}\colon ~R_{n_1,i}\subset K_0\}.\]
For $j\ge1$, suppose $K_j$ has been defined, then
\[K_{j+1}=\bigcup_{i\in \mathcal{C}_{l+1}}R_{n_{j+1},i}, \]
where
\[\mathcal{C}_{j+1}:=\{i\in\mathcal{A}_{n_{j+1}}\colon R_{n_{j+1},i}\subset K_j\}.\]
Here choose $\{n_j\}_{j\ge0}$ such that 
\begin{equation}\label{limit0}
\lim_{j\to\infty}\frac{\sum_{i=1}^{j-1}n_i}{n_j}=0,
\end{equation}
and for $j\ge1$,
\begin{equation}\label{nj}
\psi(n_j)(\lambda_d^{n_j}-1)^{-1} > (\lambda_1^{n_{j+1}}-1)^{-1}. %(not ~true ~for~ any~ j).
\end{equation}%意思是N_j阶的椭圆的最短边比\lambda_1^{-N_{j+1}大，利用引理可知，这个椭圆里有周期点，且个数不少。
Notice that $\mathcal{C}_j\ne\varnothing$ is guaranteed by  \eqref{nj}. Then
 $\{K_j\}_j$ is a decreasing sequence of nonempty sets. Let $K=\bigcap_j K_j$.

\subsection{Construct the mass distribution}

Now we define a mass distribution supported on $K$. Given a set $E\subset \mathbb{T}^d$, for $j\ge1$, denote  the collection of ellipses in $\{R_{n_j,m}\}_{m=1}^{n_j}$ which are contained in $E$ by 
$\mathcal{C}_j(E)$, that is,
$$\mathcal{C}_j(E):=\{1\le k\le H_{n_j}\colon R_{n_j,k}\subset E\}.$$

For $j=0$, $\mu(K_0)=1$. For $j=1$ and $i\in\mathcal{C}_1$,
\[\mu(R_{n_1,i})=\frac{1}{\#\mathcal{C}_1}.\]
For $j=2$ and $i\in\mathcal{C}_2$, there is a unique $m(i)\in \mathcal{C}_1$ such that $R_{n_2,i}\subset R_{n_1,m(i)}$, then let
\[\mu(R_{n_2,i})=\frac{1}{\#\mathcal{C}_2( R_{n_1,m(i)} )}\mu(R_{n_1,m(i)}).\]
%\sout{Here $\#\mathcal{C}_2( R_{n_1,m(i)}) $ is the number of ellipsoids of degree $n_2$ which are contained in $R_{n_1,m(i)}$.}
Assume that we have defined $\mu$ on the sets $\{R_{n_k,i}\colon i\in\mathcal{C}_k\}$, now for $j=k+1$ and $i\in\mathcal{C}_{k+1}$, since $\{R_{n,i}\}_{i=1}^{H_n}$ do not intersect each other,  there is a unique $m(i)\in \mathcal{C}_k$ such that $R_{n_{k+1},i}\subset R_{n_k,m(i)}$, then let
\[\mu(R_{n_{k+1},i})=\frac{1}{\#\mathcal{C}_{k+1}( R_{n_k,m(i)}) }\mu(R_{n_k,m(i)}).\]
Note that  for $j\ge1$
\[\mu(K_j)=1.\]
%and for any $m\in\mathcal{C}_j$
%\[\mu_{k+j}(R_{n_j,m})=\mu_j(R_{n_j,m}).\]
%Then we have a sequence of  measures $\{\mu_j\}_j$ supported on $K_j$, so that it has a weakly convergent subsequence. Denote the limit by $\mu$, and notice that it is by construction supported on $K$. %In fact, μn(Bi) =μn+k(Bi) for all k ≥ 0, for all Bi ∈ Fn, and similarly for ˜Ei ⊂ Bi ∈ Fn.
Then the definition of $\mu$ may be extended to all subsets of $\mathbb{T}^d$ so that $\mu$ becomes a measure. The support of $\mu$ is contained in $K$.

\subsection{ Estimation on $\mu(R_{n_j,m})$.}

%If $j=2$, $\mu(R_{n_2,m})=\frac{1}{\#\mathcal{C}_2}$, then $\log \mu(R_{n_1,m})=-\log \#\mathcal{C}_2=-\log H_{n_2}$.

%Assume that the equation holds for $j$.  Then for any $m\in\mathcal{C}_{j+1}$
%\begin{equation}
%\begin{split}
%\log \mu(R_{n_{j+1},m})&=\log\frac{1}{\#\mathcal{C}_{j+1}\cap R_{n_j,k_m} }\mu(R_{n_j,k_m})=\log \mu(R_{n_j,k_m})-\log\#\mathcal{C}_{j+1}\cap R_{n_j,k_m}\\
%&=-\log H_{n_j}+2\alpha\sum_{i=1}^{j-1}n_i+O(j)-\log\#\mathcal{C}_{j+1}\cap R_{n_j,k_m}
%\end{split}
%\end{equation}
%接下来估计一个第n_j阶的椭圆里可以包含的n_{j+1}阶椭圆的个数
%Note that $R_{n_j,k_m}$ can be covered by $\lambda_2^{2n_j}r_{n_j}^{-2}H_{n_j}^{-1}$  balls with radius $c_1^{-1}\lambda_2^{-n_j}r_{n_j}$, and contains $\lambda_2^{2n_j}r_{n_j}^{-2}H_{n_j}^{-1}$ disjoint balls with radius $c_1^{-1}\lambda_2^{-n_j}r_{n_j}$
For $j\ge1$ and $m\in\{1,\dots,H_{n_j}\}$, there exists $\{m_k\}_{k=1}^{j-1}$ with $m_k\in\{1,2, \dots,H_{n_k}\}, ~1\le k\le j-1$ such that 
$$R_{n_j,m}\subset R_{n_{j-1},m_{j-1}}\subset \dots\subset R_{n_1,m_1}.$$
Then 
$$\mu(R_{n_j,m})=\frac{1}{\#\mathcal{C}_1}\prod_{k=1}^{j-1}\frac{1}{\#\mathcal{C}_{k+1}(R_{n_k,m_k})},$$
here %When $k=0$, $R_{n_0,m_0}=\T^d$, hence 
$\#\mathcal{C}_1=H_{n_1}.$
 By Corollary \ref{cor}, there is a constant $C_1>1$ such that for any $k\ge1$, 
\[ C_1^{-1} \psi(n_k)^d\frac{H_{n_{k+1}}}{H_{n_k}}\le \#\mathcal{C}_{k+1}(R_{n_k,m_k})\le C_1 \psi(n_k)^d\frac{H_{n_{k+1}}}{H_{n_k}}.\]
It follows that  
\begin{equation}\label{measureu}
%\begin{split}
\mu(R_{n_j,m})\le C_1^{j-1} \frac{1}{\# \mathcal{C}_1}\prod_{k=1}^{j-1}\frac{H_{n_k}}{\psi(n_k)^dH_{n_{k+1}}}= C_1^{j-1} H_{n_j}^{-1}\prod_{k=1}^{j-1}\psi(n_k)^{-d}.
%&=H_{n_j}^{-1}\exp\{d\alpha\sum_{k=1}^{j-1}n_k\}.
%\end{split}
\end{equation}
Also
\begin{equation}\label{measurel}
%\begin{split}
\mu(R_{n_j,m})\ge C_1^{-j+1} H_{n_j}^{-1}\prod_{k=1}^{j-1}\psi(n_k)^{-d}.
%&=H_{n_j}^{-1}\exp\{d\alpha\sum_{k=1}^{j-1}n_k\}.
%\end{split}
\end{equation}
%Then 
%\begin{equation*}
%\begin{split}
%\log \mu(R_{n_{j+1},m})&=-\log H_{n_j}+2\alpha\sum_{i=1}^{j-1}n_i+O(j)+\log H_{n_j}+2\alpha n_j-\log %H_{n_{j+1}}+O(1)\\
%&=-\log H_{n_{j+1}}+2\alpha\sum_{i=1}^{j}n_i+O(j+1).
%\end{split}
%\end{equation*}

\subsection{Estimate the local dimension of $\mu$}

\begin{lem}\label{capwithball}
%Given balls 
Let $B := B(x, r_1)$ and  $\widetilde{B} := B(y, r_2)$. Then for $n\ge1$, the set $B \cap (A^n-I)^{-1}\widetilde{B}$ can be covered by
$$\prod_{i: r_1\le (\lambda_i^n-1)^{-1}r_2}\lceil\frac{r_1 }{r_3}\rceil\prod_{i: r_1> (\lambda_i^n-1)^{-1}r_2>r_3}\lceil\frac{(\lambda_i^n-1)^{-1}r_2}{r_3}\rceil$$
balls of radius $r_3\le r_1$.
\end{lem}
\begin{proof}
Note that $(A^n-I)^{-1}\widetilde{B}$ is contained in a rectangle $R$ with lengths $2c_2(\lambda_i^n-1)^{-1}r_2$, $i=1,\dots,d$, and 
\begin{equation*}
\begin{split}
\{1,\dots,d\}&=\{i: r_1\le 2c_2 (\lambda_i^n-1)^{-1}r_2\}\cup\{i: r_1> 2c_2(\lambda_i^n-1)^{-1}r_2>r_3\}\\
&~~\quad\cup\{i: r_3 >2c_2(\lambda_i^n-1)^{-1}r_2\}\\
&=:I_1\cup I_2\cup I_3.
\end{split}
\end{equation*}
Take
  $$
L_i=\left\{\begin{array}{ll}
 2r_1 & \text{if  \ }i\in I_1\,,\\[2ex]
 2c_2r_2(\lambda_i^n-1)^{-1} & \text{if  \ }i\in I_2\cup I_3.
                         \end{array}
 \right.
  $$
%For $i$, if $i\in I_1$, then we take $l_i=2r$, if $i\in I_2$, then we take $l_i=2(\lambda_i^n-1)^{-1}r_1$, otherwise $l_i=s$.
Then $(A^n-I)^{-1}\widetilde{B}\cap B$ can be covered by a rectangle $\widetilde{R}$ with length $L_i$, $i=1,\dots,d$. It implies that $\widetilde{R}$ may be covered by
\begin{equation*}
%\begin{split}
\prod_{i\in I_1\cup I_2}\lceil\frac{L_i}{2r_3}\rceil\prod_{i\in I_3}1
%\lceil\frac{2(\lambda_i^n-1)^{-1}r_1}{s}\rceil\prod_{i\in I_3}2&\le 4\prod_{i\in I_1}\lceil\frac{r}{s}\rceil\prod_{i\in I_2}\lceil\frac{(\lambda_i^n-1)^{-1}r_1}{s}\rceil\prod_{i\in I_3}1\\
\asymp \prod_{i: r_1\le (\lambda_i^n-1)^{-1}r_2}\lceil\frac{r_1 }{r_3}\rceil\prod_{i: r_1> (\lambda_i^n-1)^{-1}r_2>r_3}\lceil\frac{(\lambda_i^n-1)^{-1}r_2}{r_3}\rceil%\prod_{i\in I_1}\lceil\frac{r_1 }{r_3}\rceil\prod_{i\in I_2}\lceil\frac{(\lambda_i^n-1)^{-1}r_2}{r_3}\rceil.
%\end{split}
\end{equation*}
squares with length $r_3$.

\end{proof}

For $x\in K$, $x\in K_j$ for $j\ge1$, and by the construction of $K_j$, there is a unique sequence of ellipsoids containing $x$, denoted by $\{E_j\}_j$. For $r>0$, there is some $j$ such that
\begin{equation}\label{j}
(\lambda_d^{n_j}-1)^{-1}\psi(n_j)\le r<(\lambda_d^{n_{j-1}}-1)^{-1}\psi(n_{j-1}),
\end{equation}
since $\psi(n_j)(\lambda_d^{n_j}-1)^{-1}$ is decreasing as $j\to\infty$.  And there are two cases:
\begin{itemize}
\item[ \rm  \textbf{(A)}\ ]   $\exists k\ge1$ such that 
$$(\lambda_{k+1}^{n_j}-1)^{-1}\psi(n_j)\le r<(\lambda_k^{n_j}-1)^{-1}\psi(n_j),$$
\item[\rm \textbf{(B)}\ ]  $(\lambda_1^{n_j}-1)^{-1}\psi(n_j)\le r.$
\end{itemize}

First we consider Case {\rm \textbf{(A)}\ }:  $(\lambda_{k+1}^{n_j}-1)^{-1}\psi(n_j)\le r<(\lambda_k^{n_j}-1)^{-1}\psi(n_j)$.

Let $B=B(x,r)$. For $j$  in \eqref{j}, by \eqref{nj} and \eqref{j}, we have ${\rm diam}(R_{n_{j+1},m})=2(\lambda_1^{n_{j+1}}-1)^{-1}\psi(n_{j+1})<(\lambda_d^{n_{j}}-1)^{-1}\psi(n_{j})<r$, hence up to a set of zero measure,
$$B\subset \bigcup_{B\cap R_{n_{j+1},m}\ne\varnothing\atop m\in \mathcal{C}_{j+1}}R_{n_{j+1},m}\subset \bigcup_{R_{n_{j+1},m}\subset 2B\atop m\in \mathcal{C}_{j+1}}R_{n_{j+1},m}.$$

Then 
\begin{equation}\label{mub}
\begin{split}
\mu(B)&\le \sum_{R_{n_{j+1},m}\subset 2B\atop m\in \mathcal{C}_{j+1}}\mu(R_{n_{j+1},m}).
\end{split}
\end{equation}
%Now I estimate the number of ellipsoids with degree $n_j$ intersect $2B$.
 Since $r<(\lambda_1^{n_j}-1)^{-1}\psi(n_j)$,  by Remark \ref{remark} and  $\alpha>\log(\lambda_d/\lambda_1)$, %then  for any $i=1,\dots,d$ , 
 %$$(\lambda_i^n-1)r<(\lambda_i^n-1)(\lambda_1^{n_j}-1)^{-1}r_{n_j}\lesssim r_{n_j}(\lambda_d/\lambda_1)^{n_j}=\exp\{n_j(-\alpha+\log(\lambda_d/\lambda_1))\}\downarrow0.$$
 $E_j$ is the unique ellipsoid of degree $n_j$ which intersects $2B$. %In fact, it holds for any $\alpha>\log(\lambda_d/\lambda_1)$. 
 Then
\[\#\mathcal{C}_{j+1}( 2B)= \#\mathcal{C}_{j+1}(2B\cap E_j),\]
which implies that
\begin{equation*}
\begin{split}
\mu(B)&\le  (\#\mathcal{C}_{j+1}(2B\cap E_j) )\mu(R_{n_{j+1},m})\\
&\le  (\#\mathcal{C}_{j+1}(2B\cap E_j ) )\frac{\mu(E_j)}{\#\mathcal{C}_{j+1}(E_j)}
\end{split}
\end{equation*}
Now we estimate $\#\mathcal{C}_{j+1}(2B\cap E_j )$. % and $\#\mathcal{C}_{j+1}(E_j)$.
 By Lemma \ref{capwithball}, $2B\cap E_j$ can be covered by 
\[C\prod_{i: 2r\le (\lambda_i^{n_j}-1)^{-1}\psi(n_j)}\lceil\frac{2r}{\ell}\rceil\prod_{i: 2r> (\lambda_i^{n_j}-1)^{-1}\psi(n_j)>\ell}\lceil\frac{(\lambda_i^{n_j}-1)^{-1}\psi(n_j)}{\ell}\rceil\]
balls with radius $\ell$, and by Lemma \ref{local counting}, if $\ell>(\lambda_1^{n_{j+1}}-1)^{-1}$, then any ball with radius $\ell$ contains about $\ell^dH_{n_{j+1}}$ periodic points in $\mathcal{P}_{n_{j+1}}$.
Take $\ell= \psi(n_j)(\lambda_d^{n_j}-1)^{-1}$, then
\begin{equation*}
\begin{split}
\#\mathcal{C}_{j+1}(2B\cap E_j )&\lesssim \ell^dH_{n_{j+1}}\prod_{i: 2r< (\lambda_i^{n_j}-1)^{-1}\psi(n_j)}\lceil\frac{r}{\ell}\rceil\prod_{i: 2r\ge (\lambda_i^{n_j}-1)^{-1}\psi(n_j)>\ell}\lceil\frac{(\lambda_i^{n_j}-1)^{-1}\psi(n_j)}{\ell}\rceil\\
%&\lesssim |H_{n_{j+1}}|r(\lambda_2^{n_j}-1)^{-1}r_{n_j}
&= \psi(n_j)^d(\lambda_d^{n_j}-1)^{-d}H_{n_{j+1}}\prod_{i=1}^k\frac{r(\lambda_d^{n_j}-1)}{\psi(n_j)}\prod_{i=k+1}^d\frac{(\lambda_d^{n_j}-1)}{(\lambda_i^{n_j}-1)}
\end{split}
\end{equation*}
By Corollary \ref{cor},
$$\#\mathcal{C}_{j+1}(E_j)\asymp \psi(n_j)^dH_{n_{j+1}}H_{n_j}^{-1}.$$
It follows that 
\begin{equation}\label{cjbej}
\frac{\#\mathcal{C}_{j+1}(2B\cap E_j )}{\#\mathcal{C}_{j+1}(E_j)}\lesssim \prod_{i=1}^k\frac{r(\lambda_i^{n_j}-1)}{\psi(n_j)}.
\end{equation}
%and 
%\[\frac{\#\mathcal{C}_{j+1}\cap 2B\cap E_j}{\#\mathcal{C}_{j+1}\cap E_j}\lesssim \frac{H_{n_j}}{(\lambda_2^{n_j}-1)^{2}}\prod_{i: 2r< (\lambda_i^{n_j}-1)^{-1}r_{n_j}}\lceil\frac{r}{r_{n_j}(\lambda_2^{n_j}-1)^{-1}}\rceil\prod_{i: 2r\ge (\lambda_i^{n_j}-1)^{-1}r_{n_j}}\lceil\frac{(\lambda_i^{n_j}-1)^{-1}}{(\lambda_2^{n_j}-1)^{-1}}\rceil\]
%Since $(\lambda_1^{n_j}-1)^{-1}r_{n_j}> r\ge (\lambda_2^{n_j}-1)^{-1}r_{n_j}$,
%\[\frac{\#\mathcal{C}_{j+1}\cap 2B\cap E_j}{\#\mathcal{C}_{j+1}\cap E_j}\lesssim \frac{r}{r_{n_j}}(\lambda_1^{n_j}-1).\]
Combining the assumption that $\psi(n)=e^{-\alpha n}$, it derives from these estimates,  \eqref{measureu} and \eqref{measurel} that
\begin{equation*}
\begin{split}
\frac{\log\mu(B)}{\log r}&\ge  \frac{1}{\log r}\Big( \log\mu(E_j)+\log \frac{\#\mathcal{C}_{j+1}(2B\cap E_j )}{\#\mathcal{C}_{j+1}(E_j)}\Big)\\
&%\gtrsim  
\ge\frac{1}{\log r}\Big( -\log H_{n_j}+d\alpha\sum_{i=1}^{j-1}n_i +k\log r +\sum_{i=1}^k\log(\lambda_i^{n_j}-1)-\sum_{i=1}^k\log \psi(n_j)+O(j)\Big)\\
&= k+\frac{1}{\log r}\Big( d\alpha\sum_{i=1}^{j-1}n_i  -\sum_{i=k+1}^d\log(\lambda_i^{n_j}-1)+\sum_{i=1}^k\alpha n_j+O(j)\Big)\\
%1+\frac{1}{\log r}( \log (\lambda_2^{n_j}-1)^{-1}r_{n_j}+2\alpha n_j+2\alpha\sum_{i=1}^jn_i+O(j+1))\\
&= k+\frac{1}{\log r}\Big(d\alpha\sum_{i=1}^{j-1}n_i  -n_j\sum_{i=k+1}^d\log\lambda_i+k\alpha n_j+O(j+1)\Big).
\end{split}
\end{equation*}
Notice that if $k\alpha-\sum_{i=k+1}^d\log\lambda_i>0$, we have %\sout{since $\alpha>\sum_{i=2}^d\log\lambda_i$. Then }
\begin{equation*}
\begin{split}
\frac{\log\mu(B)}{\log r}&\ge k+\frac{n_j(k\alpha-\sum_{i=k+1}^d\log\lambda_i)}{\log ((\lambda_k^{n_j}-1)^{-1}\psi(n_j))}+\frac{1}{\log r}\big(d\alpha\sum_{i=1}^{j-1}n_i+ O(j+1)\big)\\
\end{split}
\end{equation*}
Note that
For $\epsilon >0$, there is some constant $j_1>1$ such that $j\ge j_1$
\[\frac{n_j(k\alpha-\sum_{i=k+1}^d\log\lambda_i)}{\log ((\lambda_k^{n_j}-1)^{-1}\psi(n_j))}\ge -\frac{k\alpha-\sum_{i=k+1}^d\log\lambda_i}{\alpha+\log\lambda_k}-\frac{\epsilon}{2}.\]
and by equation \eqref{limit0},
\begin{equation*}
\begin{split}
\frac{1}{\log r}\big(d\alpha\sum_{i=1}^{j-1}n_i+O(j+1)\big)&\gtrsim -\frac{1}{n_j(\alpha+\log\lambda_k)}\big(d\alpha\sum_{i=1}^{j-1}n_i+O(j+1)\big)\\
\end{split}
\end{equation*}
 increasingly tends to 0, as $j\to\infty$, hence there is some constant $j_2>1$ such that for all $j\ge j_2$, 
\[\frac{1}{\log r}\Big(d\alpha\sum_{i=1}^{j-1}n_i+O(j+1)\Big)\ge-\frac{\epsilon}{2}.\]
Combining these inequalities, we have
\begin{equation}\label{ek}
\frac{\log\mu(B)}{\log r}\ge  \frac{k\log\lambda_k+\sum_{i=k+1}^d\log\lambda_i}{\alpha+\log\lambda_k}-\epsilon
\end{equation}
for all $j>\max\{j_1,j_2\}$. %letting $\epsilon\to0$ %and applying the mass distribution principle, we have 
%\[\dim_{\rm H}R\ge \frac{\log\lambda_1\lambda_2}{\alpha+\log\lambda_1}.\]

 If $k\alpha-\sum_{i=k+1}^d\log\lambda_i\le0$, then for $j$ large enough, we get
\begin{equation}\label{ekp1}
\begin{split}
\frac{\log\mu(B)}{\log r}&\ge  k+\frac{1}{\log((\lambda_{k+1}^{n_j}-1)^{-1}\psi(n_j)}\Big(k\alpha n_j-n_j\sum_{i=k+1}^d\log\lambda_i\Big)-\frac{\epsilon}{2}\\
&\ge k-\frac{k\alpha-\sum_{i=k+1}^d\log\lambda_i}{\alpha+\log\lambda_{k+1}}-\epsilon\\
&= \frac{(k+1)\log\lambda_{k+1}-\sum_{i=k+2}^d\log\lambda_i}{\alpha+\log\lambda_{k+1}}-\epsilon.
\end{split}
\end{equation}
It follows from inequalities \eqref{ek} and \eqref{ekp1} that
\[\frac{\log\mu(B)}{\log r}\ge\min\Big\{ \frac{(k+1)\log\lambda_{k+1}-\sum_{i=k+2}^d\log\lambda_i}{\alpha+\log\lambda_{k+1}}, \frac{k\log\lambda_k-\sum_{i=k+1}^d\log\lambda_i}{\alpha+\log\lambda_k}\Big\}-\epsilon.\]

  Now we consider Case ({\rm \textbf{B}}):  $(\lambda_1^{n_j}-1)^{-1}\psi(n_j)\le r<(\lambda_d^{n_{j-1}}-1)^{-1}\psi(n_{j-1})$.\\
Since $2B$ does not intersect ellipsoids of degree $n_{j-1}$ except $E_{j-1}$%(By Lemma 2.6 and $\alpha>\log\lambda_d/\lambda_1$)
, we have 
\[\mu(B)\le \sum_{m\in\mathcal{C}_j\atop R_{n_j,m}\cap B\ne\varnothing}\mu(R_{n_j,m})\le \sum_{m\in\mathcal{C}_j\atop R_{n_j,m}\subset 2B}\mu(R_{n_j,m})\le (\#\mathcal{C}_j(2B\cap E_{j-1}))\mu(R_{n_j,m}).\]
There are three cases to consider.

{\bf Case (i) :}\, $r\ge(\lambda_1^{n_j}-1)^{-1}$.\\
Applying Lemma \ref{local counting}, $\#\mathcal{C}_j(2B\cap E_{j-1})\asymp r^dH_{n_j}.$ By \eqref{measureu} and \eqref{measurel},
\[\log\mu(R_{n_j,m})=-\log H_{n_j}+d\alpha \sum_{i=1}^{j-1}n_i+O(j). \]
Therefore 
\begin{equation*}
\begin{split}
\frac{\log\mu(B)}{\log r}&\ge d+\frac{1}{\log r}\Big(d\alpha\sum_{i=1}^{j-1}n_i+O(j)\Big)\\
&\ge d+\frac{d\alpha\sum_{i=1}^{j-1}n_i+O(j)}{\log \lambda_d^{-n_{j-1}}\psi(n_{j-1})}\\
&\ge d-\frac{d\alpha}{\alpha+\log\lambda_d}-\epsilon\\
&=\frac{d\log\lambda_d}{\alpha+\log\lambda_d}-\epsilon.
\end{split}
\end{equation*}
%Then we have $\dim_{\rm H}R\ge \frac{2\log\lambda_2}{\alpha+\log\lambda_2}$

{\bf Case (ii) :}\,  $(\lambda_1^{n_j}-1)^{-1}\psi(n_j)\le r<(\lambda_d^{n_j}-1)^{-1}$.\\

 By Lemma \ref{local counting}, $\#\mathcal{C}_j(2B\cap E_{j-1})\le  1 $, %$\#\mathcal{C}_j\cap 2B\cap E_{j-1}\le\prod_{i:(\lambda_i^{n_j}-1)r>1/8}\lceil (\lambda_i^{n_j}-1)r \rceil .$ Since $r\le (\lambda_1^{n_j}-1)^{-1}$, 
%\[\#\mathcal{C}_j(2B\cap E_{j-1})\lesssim \prod_{i=2}^d(\lambda_i^{n_j}-1)r,\]
hence $\mu(B)\le  \mu(E_j)$ which derives that for $j$ large enough
 \begin{equation*}
\begin{split}
\frac{\log\mu(B)}{\log r}&\ge \frac{1}{\log r}(-\log H_{n_j}+d\alpha\sum_{i=1}^{j-1}n_i+O(j))\\
%&=\frac{1}{\log r}(-\log(\lambda_1^{n_j}-1)+d\alpha\sum_{i=1}^{j-1}n_i+O(j))\\
&\ge \frac{\sum_{i=1}^d\log(\lambda_i^{n_j}-1)}{\log((\lambda_1^{n_j}-1)^{-1}\psi(n_j))}-\epsilon\\
&\ge \frac{\sum_{i=1}^d\log\lambda_i}{\log\lambda_1+\alpha}-\epsilon.
\end{split}
\end{equation*}
%which is larger than $\frac{\sum_{i=1}^d\log\lambda_i}{\alpha+\log\lambda_1}-\epsilon$,  since $\alpha>\sum_{i=2}^d\log\lambda_i$.

{\bf Case (iii) :}\, There exists $1\le k\le d-1$ such that 
\[(\lambda_{k+1}^{n_j}-1)^{-1}\le r<(\lambda_k^{n_j}-1)^{-1}.\]
By Lemma \ref{local counting}, $\#\mathcal{C}_j(2B\cap E_{j-1})\lesssim \prod_{i=k+1}^d(\lambda_i^{n_j}-1)r$, which derives that 
$$\mu(B)\le \mu(R_{n_j,m})\prod_{i=k+1}^d(\lambda_i^{n_j}-1)r.$$
Then 
\begin{equation*}
\begin{split}
\frac{\log\mu(B)}{\log r}&\ge \frac{1}{\log r}\Big(-\log H_{n_j}+d\alpha\sum_{i=1}^{j-1}n_i+O(j)+\sum_{i=k+1}^d\log(\lambda_1^{n_j}-1)+(d-k)\log r\Big)\\
& =d-k+\frac{1}{\log r}\Big(-\sum_{i=1}^k\log (\lambda_i^{n_j}-1) +d\alpha\sum_{i=1}^{j-1}n_i+O(j)\Big)\\
&\ge \frac{(d-k)(\alpha+\log\lambda_1)+\sum_{i=1}^k\log\lambda_i}{\alpha+\log\lambda_1}-\epsilon \\
&=\frac{\sum_{i=1}^d\log\lambda_i+\sum_{i=k+1}^d(\alpha-\log(\lambda_i/\lambda_1))}{\alpha+\log\lambda_1}-\epsilon\ge \frac{\sum_{i=1}^d\log\lambda_i}{\alpha+\log\lambda_1} -\epsilon,\\
\end{split}
\end{equation*}
where the last inequality holds since $\alpha>\log(\lambda_d/\lambda_1)$.

Combing Case (\textbf{A}) and Case (\textbf{B}),  for $j$ large enough, we have
\[\frac{\log\mu(B)}{\log r}\ge \min_k\Big\{\frac{k\log\lambda_k+\sum_{i=k+1}^d\lambda_i}{\alpha+\log\lambda_k}-\epsilon\Big\}.\]
letting $\epsilon\to0$ and applying the mass distribution principle \cite[Proposition 2.3]{falconer}, we have 
\begin{equation}\label{dimensionformula}
\dim_{\rm H}R(\psi)\ge \min_k\Big\{\frac{k\log\lambda_k+\sum_{i=k+1}^d\lambda_i}{\alpha+\log\lambda_k}\Big\}.
\end{equation}

Now we deal with the case that $\alpha=\log(\lambda_d/\lambda_1)$. Given $\eta>\log(\lambda_d/\lambda_1)$, define $\psi_{\eta}:\ \mathbb{R}^+\to \mathbb{R}^+$ as
\[\psi_{\eta}\colon x\mapsto e^{-x\eta}. \]
There exists an infinite set $\mathcal{N}_{\eta}\subset \mathbb{N}$ such that 
\[\psi(n)>\psi_{\eta}(n)=e^{-n\eta},\quad n\in \mathcal{N}_{\eta}.\]
Let 
\[W_{\eta}:=\{x\in\mathbb{T}^d\colon T^nx\in B(x,\psi_{\eta}(n))~\text{i.m. }n\in \mathcal{N}_{\eta}\}.\]
Note that 
\[\liminf_{n\to\infty\atop n\in \mathcal{N}_{\eta}}\frac{-\log \psi_{\eta}(n)}{n}=\eta,\]
and
\[W_{\eta}\subset R(\psi).\]
From these and inequality \eqref{dimensionformula}, we have
\[\dim_{\rm H}R(\psi)\ge \dim_{\rm H}W_{\eta}\ge \min_k\Big\{\frac{k\log\lambda_k+\sum_{i=k+1}^d\lambda_i}{\eta+\log\lambda_k}\Big\}\]
for any $\eta>\log(\lambda_d/\lambda_1)$. Letting $\eta\to \alpha:=\log(\lambda_d/\lambda_1)$, we get
\[\dim_{\rm H}R(\psi)\ge\min_k\Big\{\frac{k\log\lambda_k+\sum_{i=k+1}^d\lambda_i}{\alpha+\log\lambda_k}\Big\}.\]

Next we consider the case $\alpha=\infty$. Given $M>0$, define $\psi_{M}:\ x\mapsto e^{-xM}$. 
Then for $n$ large enough, we have 
\[\psi(n)<\psi_M(n)=e^{-nM}.\]
Let 
\[W_M:=\{x\in\mathbb{T}^d\colon T^nx\in B(x,\psi_M(n))~\text{i.m. }n\ge1\}.\]
Since $W_M\supset R(\psi)$, by Section \ref{up}, it follows that
\[0\le \dim_{\rm H}R(\psi)\le \dim_{\rm H}W_M\le \min_k\Big\{\frac{k\log\lambda_k+\sum_{i=k+1}^d\lambda_i}{M+\log\lambda_k}\Big\}.\]
 Letting $M\to \infty$, we get
\[\dim_{\rm H}R(\psi)=0.\]

Then combining these with Section \ref{up}, we finish the proof of Theorem \ref{theorem1}.

\section{The proof of Theorem \ref{theorem2}}

 Since $A$ is diagonalizable over $\mathbb{Q}$, there is an invertible  matrix $P$ such that 
\[A=P^{-1}DP,\]
where $D=\rm{diag}\{\lambda_1,\dots,\lambda_d\}$ is a diagonal matrix. 
%Assume that $P=\begin{bmatrix} \beta_1\,\beta_2\ \cdots \ \beta_d
%\end{bmatrix}$ where $\{\beta_i\}_{i=1}^d$ are column vectors.
The matrix $P$ consists of $d^2$ rational numbers. Denote $\beta$ the least common multiple of the denominators of elements of $P$. Write $\widetilde{P}=\beta P$, and $\widetilde{P}$ is an integer matrix. Then
\[A=\widetilde{P}^{-1}D\widetilde{P}.\]

Define $f\colon \mathbb{T}^d\to\mathbb{T}^d$ as
\[f\colon x\mapsto \widetilde{P}x\pmod 1. \]
Then $f \circ T=D\pmod1\circ f$. Since $\widetilde{P}$ is nonsingular, its singular values $e_1,\dots,e_d$ are strictly larger than 0. Denote
\[e_{min}=\min\{e_1,\dots,e_d\},\quad e_{max}=\max\{e_1,\dots,e_d\}.\]
Then
\begin{equation}\label{lipschitz}
e_{min}^d|x-y|\le |f(x)-f(y)|\le e_{max}^d|x-y|.
\end{equation}
It shows that $f $ is a bi-Lipschitz mapping.

For $n\ge1$, 
\begin{equation*}
\begin{split}
R_n(\psi)&=\{x\in\mathbb{T}^d\colon (A^n-I)x\pmod 1\in B(0,\psi(n))\}\\
&= \{x\in\mathbb{T}^d\colon \widetilde{P}^{-1}(D^n-I)\widetilde{P}x\pmod1\in B(0,\psi(n))\}\\
&=f^{-1}\{y\in\mathbb{T}^d\colon (D^n-I)y\pmod1\in f(B(0,\psi(n)))\}\\
&=:f^{-1}E_n
\end{split}
\end{equation*}
By inequalities \eqref{lipschitz}, 
\[B(0,e_{min}^d\psi(n))\subset f(B(0,\psi(n)))\subset B(0,e_{max}^d\psi(n)).\]
Hence $\dim_{\rm H}\big(\limsup\limits_{n\to\infty} E_n\big)$ equals to the Hausdorff dimension of 
\[\limsup_{n\to\infty}\Big\{y\in\mathbb{T}^d\colon (D^n-I)y\pmod1\in B(0,a\psi(n))\Big\}=:\limsup_{n\to\infty}E_n'\]
for any given constant $a>0$. It implies that
\[\dim_{\rm H}\big(\limsup_{n\to\infty} R_n(\psi)\big)=\dim_{\rm H}\big(\limsup_{n\to\infty} f(R_n(\psi))\big)=\dim_{\rm H}\big(\limsup_{n\to\infty} E_n'\big).\]
Note that 
\[\liminf_{n\to\infty}\frac{-\log (a\psi(n))}{n}=\alpha.\]
Combining the fact that $D$ is an integer matrix and Theorem 1.7 in \cite{hl}, we have  
 \[\dim_{\rm H}R(\psi)=\min_{1\le j\le d}\Big\{\frac{j\log|\lambda_j|+\sum_{k\in\mathcal{K}(j)}(\alpha+\log|\lambda_j|-\log|\lambda_i|)+\sum_{i=j+1}^d\log|\lambda_i|}{{\log|\lambda_j|+\alpha}}\Big\}
,\]
where
\[\mathcal{K}(j):=\{1\le i \le d\colon \log|\lambda_i|>\log|\lambda_j|+\alpha\}.\]

\subsection*{Acknowledgements}

Z. H. was supported by Science Foundation of China University of Petroleum, Beijing (Grant No. 2462023SZBH013), and China Postdoctoral Science Foundation (Grant No. 2023M743878). B. L. was supported partially by NSFC 12271176.


\begin{thebibliography}{10}
\bibitem{abb}
D. Allen, S. Baker and B. B\'ar\'any. Recurrence rates for shifts of finite type. Preprint, 2022,
arXiv:2209.01919.

\bibitem{BF}
S. Baker and M. Farmer. Quantitative recurrence properties for self-conformal sets. Proc. Amer. Math.
Soc. 149 (2021), 1127–1138.

\bibitem{BS}
L. Barreira and B. Saussol. Hausdorff dimension of measures via Poincar\'e recurrence. Comm. Math. Phys.
219 (2001), 443–463.

\bibitem{bos}
M. D. Boshernitzan, Quantitative recurrence results, Invent. Math. 113 (1993), no. 3,
617–631.


\bibitem{CWW}
Y. Chang, M. Wu, and W. Wu,
 Quantitative recurrence properties and homogeneous
self-similar sets, Proc. Amer. Math. Soc. 147 (2019), no. 4, 1453–1465.
 
 \bibitem{CK}
 N. Chernov and D. Kleinbock, Dynamical Borel–Cantelli lemmas for Gibbs measures, Israel
J. Math. 122 (2001), 1–27

\bibitem{EW99}
G. Everest and T. Ward, 
Heights of polynomials and entropy in algebraic dynamics. Universitext Springer-Verlag London, Ltd., London, 1999. xii+211 pp.

\bibitem{falconer}
 K. J. Falconer, Techniques in Fractal Geometry, John Wiley and Sons, Chichester, 1997.

\bibitem{FMP}
J. L. Fern$\acute{{\rm a}}$ndez, M. V. Meli$\acute{{\rm a}}$n, and D. Pestana, Quantitative mixing results and inner functions,
Math. Ann. 337 (2007), no. 1, 233–251

\bibitem{hl}
Y. He and L. Liao,
Quantitative recurrence properties for piecewise expanding maps on $[0,1]^d$, Preprint, 2023, arXiv:2302.05149.

\bibitem{HV95}
R. Hill and S. Velani. The ergodic theory of shrinking targets, Invent. Math. 119
(1995), no. 1, 175–198

\bibitem{HV}
R. Hill and S. Velani. The shrinking target problem for matrix transformations of tori. J. London Math.
Soc. (2) 60 (1999), 381–398.

\bibitem{HuPer}Z. Hu and T. Persson, Hausdorff dimension of recurrence sets. Preprint, 2023,	 arXiv:2303.02594.

\bibitem{hlsw}
M. Hussain, B. Li, D. Simmons and B. Wang. Dynamical Borel-Cantelli lemma for recurrence
theory. Ergodic Theory Dynam. Systems 42 (2022), 1994-2008.

\bibitem{llvz}
B. Li, L. Liao, S. Velani and E. Zorin. The shrinking target problem for matrix transformations
of tori: revisiting the standard problem. Adv. Math. 421 (2023), Paper No. 108994, 74.

\bibitem{lwwx}
B. Li, B. Wang, J. Wu, and J. Xu, The shrinking target problem in the dynamical
system of continued fractions, Proc. Lond. Math. Soc. (3) 108 (2014), no. 1, 159-186,

\bibitem{kz}
D. Kleinbock and J. Zheng. Dynamical Borel-Cantelli lemma for recurrence under Lipschitz
twists. Nonlinearity 36 (2023), 1434-1460.

\bibitem{kkp} M. Kirsebom, P. Kunde and T. Persson, On
    shrinking targets and self-returning points, to appear in
  Annali della Scuola Normale Superiore di Pisa, Classe di
  Scienze, arXiv:2003.01361v2.

%\bibitem{perssonreeve} T. Persson, H. Reeve, \emph{A Frostman   type lemma for sets with large intersections, and an   application to Diophantine approximation}, Proceedings of the  Edinburgh Mathematical Society, Volume 58, Issue 02, June 2015,  521--542.

%\bibitem{per22} T. Persson, \emph{A mass transference principle    and sets with large intersections}, Real Anal. Exchange 47  (2022), no. 1, 191--205.
\bibitem{SW}
S. Seuret and B. Wang. 
Quantitative recurrence properties in conformal iterated function systems. Adv.
Math. 280 (2015), 472–505.

\bibitem{TW}
B. Tan and B. Wang. Quantitative recurrence properties for beta-dynamical system. Adv. Math. 228 (2011),
2071–2097.

\bibitem{entropy}
P. Walters. An Introduction to Ergodic Theory, Graduate Texts in Mathematics, vol. 79, Springer-Verlag, New York-Berlin, 1982.

\bibitem{wy}
Y. Wu and N. Yuan. Recurrent set on some Bedford–McMullen carpets. Preprint, 2022,
arXiv:2209.07315.

\bibitem{yl}
N. Yuan and B. Li. Hausdorff dimensions of recurrent and shrinking target sets under Lipschitz functions for expanding Markov maps. Dyn. Syst. 38(3), 365-394 (2023), DOI: 10.1080/14689367.2023.2184328.

\bibitem{yw}
N. Yuan and S. Wang. Modified shrinking target problem for Matrix Transformations of Tori.
Preprint, 2023, arXiv:2304.07532.

\end{thebibliography}
\end{document}